\documentclass[final,5p,twocolumn]{elsarticle}

\usepackage{graphicx}          
                               
\usepackage{graphbox}          

\usepackage[dvips]{epsfig}    

\usepackage{xcolor}

\usepackage{amssymb}
\usepackage{amsthm}

\usepackage{amsmath}          

\journal{XXX}

\newtheorem{theorem}{Theorem}[section]
\newtheorem{assumption}[theorem]{Assumption}

\begin{document}

\begin{frontmatter}

\title{On Cauchy problem and stability of inversion-free feedforward control of piecewise monotonic Krasnosel\'skii-Pokrovskii hysteresis} 

\author[First]{Jana Kopfov\'a}
\ead{jana.kopfova@math.slu.cz}
\author[Second]{Michael Ruderman}
\ead{michael.ruderman@uia.no}

\address[First]{Mathematical Institute of the Silesian University, Czech Republic}
\address[Second]{Department of Engineering Sciences of the University of Agder, Norway}

\begin{abstract}                          
We consider the non-homogeneous first-order differential equation with hysteresis described by the Krasnosel\'skii-Pokrovskii rate-independent hysteresis operator. Existence and  uniqueness of solutions as well as  the boundedness of solution in response to a bounded input are proved. The global stability of the equation is also investigated. Periodic solutions and their stability are studied in addition. The differential equation under analysis constitutes the so-called inversion-free feedforward control, which was proposed for mitigating arbitrary rate-independent hysteresis effects in the actuated systems. The experimentally identified non-smooth and non-strictly monotonic hysteresis of a magnetic shape memory alloy (MSMA) actuator serves as the case study. The performed analysis is settled in a series of theorems which are illustrated by  numerical examples. 
\end{abstract}


\end{frontmatter}

\newtheorem{thm}{Theorem}
\newtheorem{lem}[thm]{Lemma}
\newtheorem{clr}{Corollary}
\newdefinition{rmk}{Remark}
\newproof{pf}{Proof}

\newcommand\red{\color{red}}
\newcommand\blue{\color{blue}}

\color{black}

\section{Introduction}
\label{INTRO}

Hysteresis phenomena arise in a wide range of physical systems, including ferromagnetic materials, piezoelectric actuators, and multi-stable functional materials such as, for example, magnetic shape memory alloys (MSMA). In many of these applications, hysteresis severely limits the achievable precision of control strategies, and its compensation is therefore an essential component in high-performance system design.

Among various analytical descriptions of hysteresis, the Krasnosel'skii-Pokrovskii (KP) operator \cite{1} represents a widely used and experimentally validated model class,  suitable for systems exhibiting piecewise monotone, rate-independent hysteresis loops.

This work is motivated by the inversion-free feedforward control framework proposed in \cite{2}, as illustrated in Figure \ref{fig:1}.
\begin{figure}[!ht]
\centering
\includegraphics[scale=0.6]{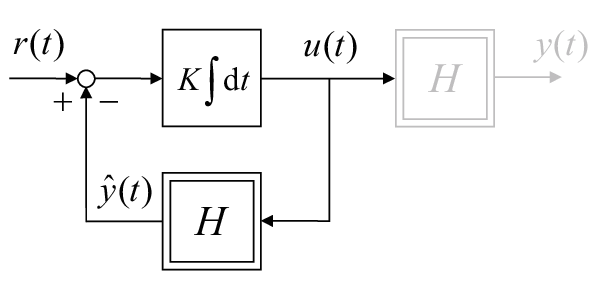}
\caption{Inversion-free feed-forward hysteresis control.}
\label{fig:1}       
\end{figure} 
Instead of explicitly computing the inverse hysteresis map, the approach uses a model feedback loop to implicitly generate the compensating signal $u$. The underlying dynamical system is described by the differential equation
\begin{equation}\label{eq:main}
\frac{1}{K}\frac{du(t)}{dt} + H(u(t)) = r(t), \qquad t \ge 0,
\end{equation}
with initial conditions \(u(0)=u_0\), $H(u)(0) = w_0,$ design parameter \(K>0\), and a continuous input \(r(t)\). For sufficiently large gain \(K\), the solution \(u(t)\) approximates the action of the inverse hysteresis operator \(H^{-1}\), and thus acts as a feedforward compensation signal for a hysteretic part, see the gray-shadowed right-hand side of Fig.~\ref{fig:1}. The performance of the scheme is reflected in the tracking error norm \( |r(t) - \hat{y}(t)| \), with \( \hat{y}(t) = H(u(t)) \).

While various control schemes are applicable and have been studied for MSMA in the last two decades, e.g. \cite{5}, \cite{7}, a pre-compensation of a large input hysteresis remains one of the main control issues. An example of the measured input-output response, \cite{6}, of a prototypic MSMA actuator, \cite{5}, under quasi-static excitation conditions is shown in Figure \ref{fig:2} in comparison with the identified hysteresis model which assumes three elementary operators, each of them of a KP-type. Note that here the quasi-static conditions mean that the applied actuator excitation current had a form of a low linear slope, i.e. with a sufficiently low $|du/dt|$ value. Under such conditions, an input-output hysteresis response of the actuator can be considered without any other actuator dynamics, like those due to inertial, damping, or stiffness elements. We then assume a rate-independent hysteresis map $y = H(u)$, while feedforward compensation scheme as in Figure \ref{fig:1} is the main objective, reflected in \eqref{eq:main}.
\begin{figure}[!ht]
\centering
\includegraphics[scale=0.6]{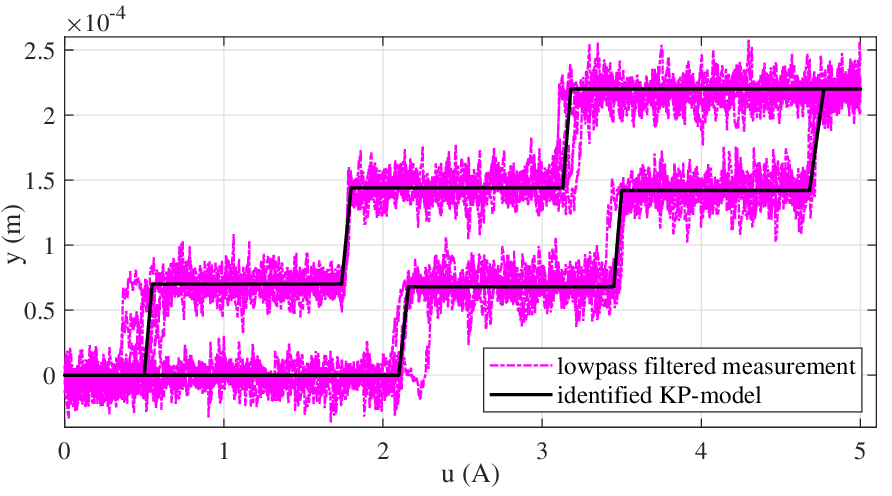}
\caption{Measured (and lowpass filtered) MSMA actuator response under quasi-static operation conditions versus the identified KP-type hysteresis model with three operator elements.}
\label{fig:2}       
\end{figure} 

A central problem  addressed in this paper is the analysis of the well-posedness for \eqref{eq:main}, particularly the existence and uniqueness of solutions, 
existence of periodic solutions, and their long term behavior, i.e. the  stability of steady state solutions and  periodic solutions. The hysteresis nonlinearity, and also  its non-strict monotonicity are features that complicate the analyses.


In the following, we extend the inversion-free hysteresis compensation framework of \cite{2,6} by developing a rigorous analysis for a broad class of hysteresis maps. More precisely:
\begin{itemize}
    \item We introduce the class of \emph{generalized play operators}, which includes the classical KP-type hysteresis operators as a special case.
    \item Under physically reasonable assumptions, we establish \emph{existence} and \emph{uniqueness} of solutions to the nonlinear differential equation \eqref{eq:main} with hysteresis.
    \item We derive \emph{a priori bounds} on the solution and prove stability properties of the associated dynamical system.
    \item We investigate the existence of \emph{periodic solutions} to \eqref{eq:main} under periodic forcing \(r(t)\), and we provide sufficient conditions for their stability.
    \item Numerical simulations illustrate the theoretical findings and demonstrate also the effectiveness of inversion-free feedforward compensation for KP-type hysteresis.
\end{itemize}

\noindent The results contribute to the mathematical foundations of hysteresis compensation and provide practically relevant tools for control applications involving complex multi-stable materials and actuators.

The mathematical theory of hysteresis has been extensively developed over the last decades, beginning with the classical monographs of Krasnosel'skii and Pokrovskii \cite{1} and continuing through modern analytical frameworks such as those presented by Brokate and Sprekels \cite{BrokateSprekels} and Visintin \cite{8}. These works established the fundamental properties of rate-independent operators, including play and stop operators, Preisach-type superposition models, and their structural assumptions such as monotonicity, Lipschitz continuity, or parameterized dissipation potentials. Also  a significant portion of the existing literature focuses on invertibility or strict monotonicity of hysteresis operators, as these conditions ensure well-posedness of differential equations with hysteresis terms and enable the construction of inverse operators central to compensation-based control strategies, cf. \cite{Mayergoyz}. However, KP-type hysteresis operators arising in physical applications often feature flat regions and non-invertible branches, and therefore violate these classical assumptions. The non-strict monotonicity of the KP operator  poses analytical challenges not fully covered by existing results on differential equations with hysteresis. The present work contributes to this line of research by establishing rigorous existence, uniqueness, boundedness, and stability properties of solutions to \eqref{eq:main}, under assumptions directly motivated by physical applications. However, the hysteresis models in this work are considered under more general assumptions as is reasonable to assume in physical applications. We further provide conditions guaranteeing the existence and stability of periodic solutions, thereby addressing aspects that remain only partially explored in the current mathematical literature. Numerical examples reconfirm theoretical results formulated in the series of theorems.

\section{Problem statement}
\label{PROBLEM}

We consider the class of dynamical systems with hysteresis described by the differential equation  \eqref{eq:main}
with the initial condition 
$u(0)= u_0$ and the independent (time) variable $t\in[0,\infty)$. Here $H(\cdot)$ represents a generalized play operator, described in detail in the next section. The KP-type hysteresis operators coming from applications constitute a special case of these operators. The right hand side $r(t)$ is assumed to be Lipschitz continuous. The constant $K > 0$ is the design parameter.

\color{black}

\section{Generalized play operator}

Let
$
 {\gamma  }_{l},  {\gamma  }_{r}  : \mathbb{R} \rightarrow
\mathbb{R}, \quad  \text{ with } \quad
{\gamma }_{r } \leq {\gamma }_{l }
$
be continuous non-decreasing functions. Now, given ${w}_{0}\in \mathbb{R}$, we construct the hysteresis operator $H (\cdot, {w}_{0})$ as follows. Let $u$ be any continuous,
piecewise linear function on ${\mathbb{R}}^{+}$ such that $u$ is linear on
$[{t}_{i - 1}, {t}_{i}] $ for $i = 1,2,...$
We then define $w :=H  (u, {w}_{0}):{\mathbb{R}}^{+}\rightarrow \mathbb{R}$ by
\begin{equation}\label{rv}
w(t) :=
\begin{cases}
   \min   \{{\gamma  }_{l}(u(0)),   \max  \{{\gamma  }_{r}(u(0)),
{w}_{0}\}\}&\text{ if }
t = 0,\\
   \min   \{{\gamma  }_{l}(u(t)),   \max  \{{\gamma  }_{r}(u(t)),
w({t}_{i - 1})\}\}&
\text{ if } t \in I,
\end{cases}
\end{equation}
where $I = ({t}_{i - 1},{t}_{i}], i = 1,2,\ldots$. 

We also assume that 
${\gamma }_{r}(u(0)) \leq {w}_{0}
\leq {\gamma  }_{l}(u(0)) $, i.e. $w(0) = {w}_{0}.$ 

\noindent As proved in \cite[Section~III.2]{8}, for  any continuous
piecewise  linear  functions  ${u  }_{1  }  $,  ${u  }_{2  } $ on
${\mathbb{R}  }^{+} $,  with the  notation ${\epsilon  }_{k }  :=
H ({u }_{k }, {w }_{0k }) $, $k=1,2 $, we have the following
inequality:
\begin{align}\nonumber
\max_{[{t }_{1 }, {t }_{2 }]}  |{\epsilon }_{1 } - {\epsilon }_{2
} | \leq \\ \nonumber \max\left\{ | {\epsilon }_{1  } ({t }_{1 })  - {\epsilon
}_{2 } ({t }_{1}) |, {m  }_{M }\left( \max_{ [{t }_{1 }, {t }_{2
} ]} |{u }_{1 } - {u }_{2 }| \right) \right\} \\ \label{qw} 
\forall [{t }_{1 }, {t }_{2 }] \subset [0,T], T \in {\mathbb{R} }^{+},
\end{align}
where for any  continuous function  $f:   \mathbb{R}  \rightarrow
\mathbb{R}$ and any  constant $ M >0 $,  ${|f| }_{M}(h) $ denotes
its local modulus of continuity:
\begin{align}
{|f| }_{M } (h) := \sup \bigl\{ |f({y }_{1 }) - f({y }_{2 } )| \; : \\ \nonumber  {y }_{1
}, {y  }_{2 } \in  [-M , M],  |{y }_{1 }  - {y }_{2  }| \leq h \bigr \}
\qquad  \forall h > 0 ,
\end{align}
\begin{equation}
{m }_{M } (h) := \max \{  {|{\gamma }_{l }| }_{M } (h), {|{\gamma
}_{r }| }_{M }(h) \} \qquad \forall h, M > 0,
\end{equation}
and
\begin{equation}
M := \max \{ | {u }_{k }(t)| : t \in [0,T]   , k = 1,2      \}.
\end{equation}

Hence $H (\cdot , {w}_{0})$ has an unique continuous extension, denoted by
$H  (\cdot ,  {w}_{0})$ again,  to an  operator
\begin{equation}
H:C({\mathbb{R}}^{+}) \times \mathbb{R} \rightarrow C({
\mathbb{R}}^{+ }).
\end{equation}
This operator is called a generalized play.%

The inequality (\ref{qw}) holds also for this extended operator,
which is then uniformly continuous on bounded sets. If ${\gamma
}_{l  }  $,  ${\gamma  }_{r  }  $  are Lipschitz continuous, then
$H $ is also Lipschitz continuous.

\section{Main Results}
\label{Model}

We  prove an uniqueness  theorem for the Cauchy problem with $H(u(t))$  a
generalized play operator. 
We will assume in the whole section that the Assumption \ref{H} is always satisfied:

\begin{assumption}\label{H}
$H(u,{w}_{0})(t) $ is a generalized play operator  with hysteresis boundary curves  ${\gamma}_{l}$
and ${\gamma}_{r}$ which are non-decreasing and Lipschitz continuous, $r(t)$ is Lipschitz continuous.\\
\end{assumption}

{\emph{Remark:} As a consequence of this Assumption we can assume for the rest of the paper that the considered hysteresis operator $H(u(t))$ is Lipschitz continuous.

\begin{theorem}
Suppose that Assumption \ref{H} holds. 
Then the solution of the Cauchy problem \eqref{eq:main} for any $K>0$
with initial conditions $u(0) = u_0,$ $H(u)(0) = w_0,$
exists and is unique.
\end{theorem}

\begin{proof}
The proof is a standard application of well-known results from the ODE theory, see e.g. \cite{3} for more details. The Lipschitz continuity of the hysteresis operator $H(u(t))$ is the essential property from which the result follows.
\end{proof}

The interesting question about the long time behaviour of our ODE  \eqref{eq:main}, we are answering in a few steps. We prove first the boundedness of the solution. 

\begin{theorem}
Suppose that Assumption \ref{H}  holds.
Assume, moreover, that the hysteresis term $H(u)$ as well as the function  $r(t)$
are nonnegative and bounded, i.e. 
that  there exist  constants $H>0 $ and $R_{\infty}$ such that for any $t \geq 0$, $0\leq H(u(t))\leq H  $ and moreover $0<r(t)\leq R_{\infty} <H. $\\
Then the solution $u(t)$ of the Cauchy problem  \eqref{eq:main}
with initial conditions $u(0) = u_0$, $H(u)(0) = w_0$,
is bounded for any $t \geq 0.$ This holds for any $K>0$.
\end{theorem}

\begin{proof}
The idea of the proof is very intuitive, it comes from the definition of the generalized play operator. Assume by contradiction that the solution is unbounded from above. Then there exists rightmost   point $t_1$ for which we  have $H(u(t_1))=R_{\infty}. $ Then for all $t>t_1 $ we have $H(u(t))>R_{\infty}, $ this means that $\frac{d u }{d t } (t) <0,$ which is a contradiction.

The boundedness from below is proved by the same argument.
\end{proof}

{\emph{Remark:} The assumption that both $H(u(t))$ and $r(t)$ are nonnegative is application-dependent and, in the studied case (cf. Fig. \ref{fig:2}), reflects the only unipolar actuator excitation and corresponding elongation of MSMA. However, this assumption can be easily generalized analogically to the upper bounds.

\vspace{3mm}

Assume now that $r(t) $ is constant. Then we can  prove the asymptotic stability of our ODE \eqref{eq:main}: 

\begin{theorem}
Suppose that Assumption \ref{H} holds and $r(t) = R, $ for $t \geq 0,$ where $0<R< H,$ 
Then the solution $u(t)$ of the Cauchy problem \eqref{eq:main}
with initial conditions $u(0) = u_0$, $H(u)(0) = w_0$,
is stable. Define $u_1$ to be the highest value $u$ for which   $\gamma_l (u) = R$ and  $u_2$ to be the smallest value of $u,$ for which  $\gamma_r (u) = R.$
Then, depending on the initial conditions, the solution converges to  $u_*$
 from the set $\{u_1,u_2\}.$
\end{theorem}

\begin{proof}
Following the idea from the previous proof, 
it is clear that as soon as $H(u(t)) >R,$ for some $t$ the solution is decreasing and similarly if  $H(u(t)) <R$ for some $t,$ the solution is increasing.  In both cases  the solution preserves its monotonicity till $H(u(t)) = r$. The statement follows. 
\end{proof}

Assume now that $r(t)$ is not a constant but converges as $t\rightarrow \infty$ to the limit value $R_{\infty}.$

\begin{theorem}
  Suppose that Assumption \ref{H} holds and $\lim_{t\rightarrow \infty}r(t) = R_{\infty}, $  where $0<R_{\infty}< H.$ 
Then the solution $u(t)$ of the Cauchy problem \eqref{eq:main}
with initial conditions $u(0) = u_0$, $H(u)(0) = w_0$,
is  stable. 
Define $u_1$ to be the biggest point where $\gamma_l (u) = R_{\infty}$ and  $u_2$ to be the smallest point where $\gamma_r (u) = R_{\infty}.$
Then the $\omega-$limit set of the solution $u(t)$ of \eqref{eq:main} is the set $\{u_1,u_2\}.$
\end{theorem}

\begin{proof}
The idea of the  proof is very similar to the previous one, except we need to use some $\epsilon - \delta$ machinery to handle the condition that $\lim_{t\rightarrow \infty}r(t) = R_{\infty}. $
Namely,  take arbitrary   $\epsilon >0$,
and assume that there exist $t_1>0$ such that for all $t>t_1$  \[u(t) >u_2 + \epsilon.\]

Let $\delta = \gamma_r(u(t_1))-\gamma_r(u_2). $ Let us note that consequently $\delta\leq H.$

Then, the assumption $\lim_{t\rightarrow \infty}r(t) = R_{\infty} $ implies that 
there exists some $t_{*}>0,$ such that $r(t)< R_{\infty}+\delta $ for all  $t>t_{*}.$ 

Now take $t \geq \max\{t_{*}, t_1\}.$
We have from the equation that at any such $t$ the solution $u(t) $ must be decreasing.

The case  $u(t) < u(t_1)$ can be handled   analogously.
Therefore the   $\omega-$limit set cannot lie outside the interval $(u_1 - \epsilon, u_2 + \epsilon )$ for any  $\epsilon >0$. This proves the Theorem. 
\end{proof}

Moreover, we can also show that 

\begin{theorem}
Suppose that Assumption \ref{H} holds and \(r(t)\) is \(T\)-periodic and bounded.
Then the solution $u(t)$ of the Cauchy problem \eqref{eq:main}
with initial conditions $u(0) = u_0$, $H(u)(0) = w_0$,
is   \(T\)-periodic and stable. 
\end{theorem}

\begin{proof}
We prove  existence of a periodic solution using the Poincaré map: 

Define the \emph{Poincaré map} \(P : \mathbb{R} \to \mathbb{R}\) in a standard way, 
a fixed point \(u_*\) of \(P\) then corresponds  to a \(T\)-periodic solution of our  equation \eqref{eq:main}.

Let \(u_1(t)\) and \(u_2(t)\) be two solutions of the \eqref{eq:main} with the same \(r(t)\), but with different initial conditions \(u_1(0)\) and \(u_2(0)\).  
Define
\[
w(t) := u_1(t) - u_2(t).
\]
Then
\[
\frac{dw}{dt} = -\big(H(u_1(t)) - H(u_2(t))\big).
\]
Multiplying the last equation by \(w(t)\) gives
\[
\frac12 \frac{d}{dt} (w^2(t)) = -\big(H(u_1(t)) - H(u_2(t))\big) \, w(t) \le 0
\]
by the monotonicity of \(H\).  Therefore,
\[
|w(t)| \ \text{is nonincreasing in } t.
\]
In particular,
\[
|P(u_1(0)) - P(u_2(0))| \le |u_1(0) - u_2(0)|.
\]

Boundedness of solutions (can be shown as before, see Theorem 2.2) ensures that \(P\) maps a compact convex subset of \(\mathbb{R}\) into itself.  
By the Brouwer fixed-point theorem, \(P\) has at least one fixed point, hence \eqref{eq:main} has at least one \(T\)-periodic solution \(u_p(t)\).

The Lipschitz continuity  of \(H\) implies that  the fixed point is \emph{unique}.

Let \(u(t)\) be any solution  and \(u_p(t)\) be the unique periodic solution with initial value \(u(0) = u_0\).
By repeating the arguments as above, it can be shown that the difference \(v(t) := u(t) - u_p(t)\) satisfies
\[
\frac12 \frac{d}{dt} (v^2(t)) \le 0,
\]
so \(|v(t)|\) is nonincreasing over time. 
This implies stability.
\end{proof}




\begin{theorem}{Error bound of $e=r-H(u)$.}

Suppose that Assumption \ref{H} holds. Assume, moreover that the hysteresis term $H(u)$ as well as the function  $r(t)$ are nonnegative and bounded from above, i.e. that there exist constants $H>0 $ and $R_{\infty}$ such that for any $t \geq 0$, $0\leq H(u(t))\leq H$ and moreover 
$0<r(t)\leq R_{\infty} <H$.\\
Then we have  
\[
|e(t)| \;\le\; R_{\infty} + H,
\qquad \forall t \ge 0.
\]
\end{theorem}

\begin{proof}
By definition of the generalized play operator,
for each $t \ge 0$ the output satisfies
\[
H(u)(t) \in [\gamma_\ell(u(t)),\, \gamma_r(u(t))].
\]
Therefore
\[
e(t) = -H(u)(t) + r(t) \;\in\; r(t) - [\gamma_\ell(u(t)), \gamma_r(u(t))],
\]
i.e.
\[
|e(t)| \;\le\; \max\big\{ |r(t) - \gamma_\ell(u(t))|,\; |r(t) - \gamma_r(u(t))| \big\}.
\]
Consequently using the bounds from the Assumption we get the desired estimate.
Let us note that this estimate is independent of both the gain $K$ and the trajectory $u(t)$.
\end{proof}

{\emph{Remark:} The above estimate of the error bound is the conservative one. Depending on the gain parameter value $K$ and the Lipschitz constant of the input trajectory 
the error bound can be lower to a certain extent (cf. numerical examples below). 

\section{Numerical Cases}
\label{Numeric}

We consider the problem \eqref{eq:main} and perform several numerical cases to illustrate the analysis results derived above. The implemented model \eqref{eq:main} uses a standard forward Euler numerical solver with the fixed time step assigned to $10^{-6}$ sec. The KP-type hysteresis $H(u)$ is the one identified from the experimental data, see Figure \ref{fig:1}, and its input-output implementation uses three saturated elementary Prandtl-Ishlinskii play-type operators, see e.g. \cite{9} for basics, connected in parallel. Note that in the following, for the sake of simplicity, the numerical values for $r$ and $H(u)$ are used without $\times 10^{-4}$ scaling, cf. Figure \ref{fig:1}. Two gain factor values $K = \{10, \, 50\}$ are assigned to better distinguish the resulting convergence behavior. We consider the control error
\[
e(t) = r(t) - H\bigl(u(t)\bigr) 
\]
as a convergence factor of interest. Recall that the $|e(t)|$ norm serves as the main performance metric of the feed-forward hysteresis control, cf. Section \ref{INTRO}. 

First, a constant input value $r(t) = R = 2$ is applied at time $t=0.1$ sec. The control error and the logarithm of its amplitude are shown in Figure \ref{fig:3} (a) and (b), respectively. The convergence is clearly inline with Theorems 3.4, while the exponential convergence rate depends on the gaining factor $K$.
\begin{figure}[!ht]
\centering
\includegraphics[scale=0.65]{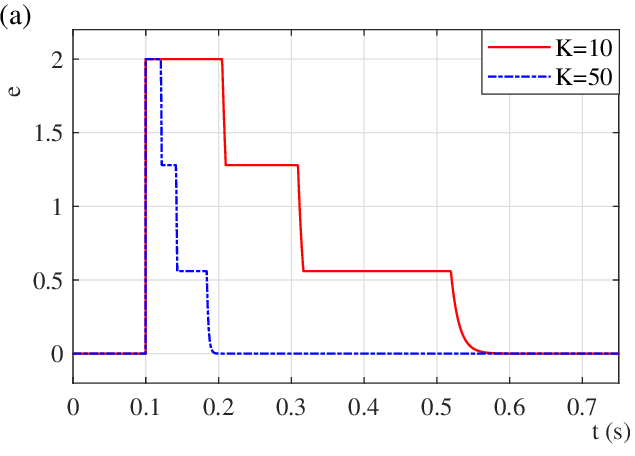}
\includegraphics[scale=0.65]{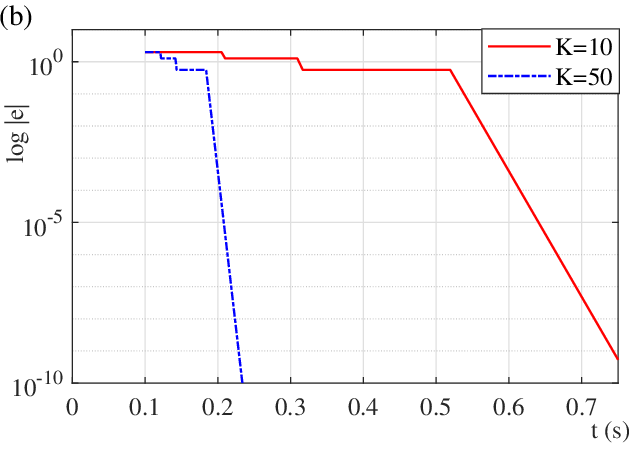}
\caption{Control error (a) and its logarithmic convergence (b) for a constant (step) input value $r(t) = R = 2$ applied at time $t=0.1$ sec.}
\label{fig:3}       
\end{figure} 

Next, a smooth input function 
\begin{equation*}
r(t) = \frac{a_1 t^n}{h^n + t^n} + \frac{a_2}{\sqrt{2 \pi \sigma^2}} \exp \biggl(-\frac{(t-\mu)^2}{2 \sigma^2} \biggr)\, \sin(\omega t), 
\end{equation*}
which satisfies $\lim_{t \rightarrow \infty} r(t) = R_{\infty} = a_1$, cf. the $r$-condition in Theorem 3.5, is used. The parameter values are assigned as $n=4$, $h=0.2$, $a_1 = 2$, $a_2 = 0.1$, $\sigma = 0.1$, $\mu = 0.3$, $\omega = 100$, resulting in the $r(t)$ trajectory as shown in Figure \ref{fig:4} (a). The corresponding control error, which is apparently converging to zero, is shown for both gaining factors $K$ in Figure \ref{fig:4} (b), in accord with Theorem 3.5. 
\begin{figure}[!ht]
\centering
\includegraphics[scale=0.65]{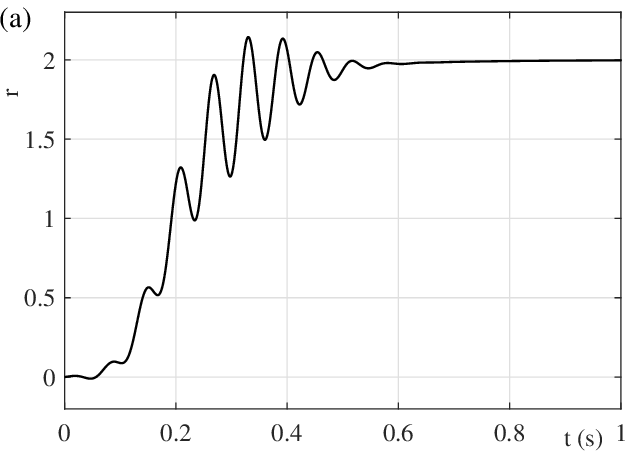}
\includegraphics[scale=0.65]{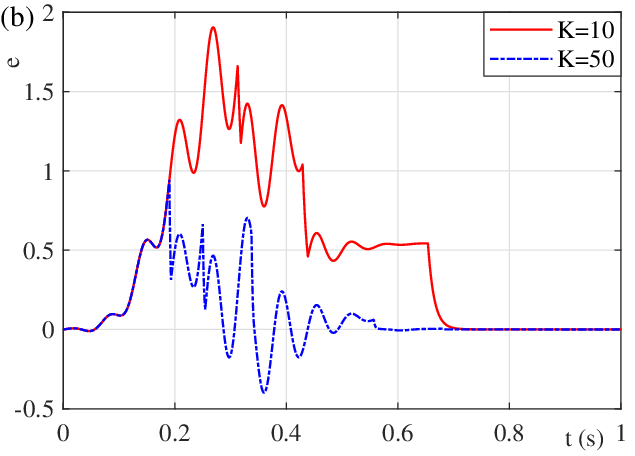}
\caption{Input trajectory (a) and control error (b) in case $\lim_{t \rightarrow \infty} r(t) = 2$.}
\label{fig:4}       
\end{figure} 

Finally, we apply a periodic input function satisfying $0 < r(t) < H$, cf. Theorem 3.6. The assumed  function 
\begin{equation*}
r(t) = A_0 + A \sin \bigl(\omega t +  \phi \bigr) 
\end{equation*}
is smooth for all $t > 0$, while its dynamics rate and so the Lipschitz constant are directly controllable by the angular frequency $\omega$. The parameters $A_0 = 1.1$, $A = 1$, and $\phi = -\pi/2$ are assumed with respect to the parameterized hysteresis $H(u)$, cf. Figure \ref{fig:1}.
\begin{figure}[!ht]
\centering
\includegraphics[scale=0.72]{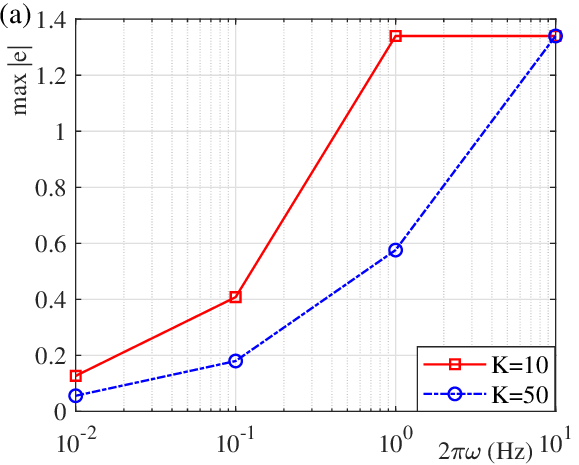}
\includegraphics[scale=0.7]{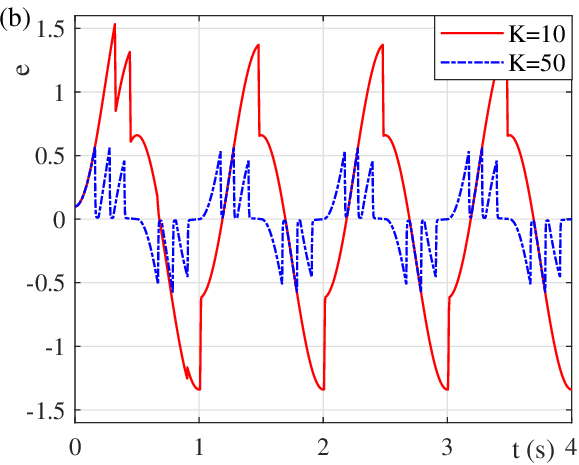}
\caption{Maximal steady-state absolute control error depending on the frequency $2 \pi \omega$ (a), and control error for $2 \pi \omega = 1$ Hz (b).}
\label{fig:5}       
\end{figure} 
The steady-state absolute control error $\max |e|_{t \rightarrow \infty}$ is evaluated for both gaining factors $K = \{10, 50 \}$. Also the set of angular frequencies $\omega = (2\pi)^{-1} \times \{   0.01,\, 0.1,\, 1,\, 10\}$ rad/sec is used for different simulations, this way covering three decades on the logarithmic scale, see Figure \ref{fig:5} (a). Note that the input frequency is depicted logarithmical in the Hz units,  for the sake of convenience. The control error itself over the time is also shown exemplary in Figure \ref{fig:5} (b) for the angular frequency $\omega = (2\pi)^{-1}$ rad/sec. We notice that the obtained numerical results depicted in Figure \ref{fig:5} (a) are qualitatively in accord with the previous approximative analysis based on the linear transfer functions provided in \cite{2}, cf. with \cite[Fig.~7]{2}.      

\color{black}

\section{Conclusions}
\label{concl}

We considered the non-homogeneous first-order differential equation with hysteresis, described by the Krasnosel\'skii-Pokrovskii rate-independent hysteresis operator. More specifically, we studied the boundedness of solution in response to a bounded input and the global stability of such equation that describes the so-called inversion-free feedforward hysteresis control \cite{2}. The problem was motivated by experimentally identified hysteresis of a magnetic shape memory alloy (MSMA) actuator, where non-smoothness and non-strict monotonicity of the hysteresis map constitute the essential concerns addressed in this work. In a series of theorems, we
analyzed the existence, uniqueness, and boundedness of  solutions of the  dynamical system with hysteresis, and the stability of the solution as well the existence, uniqueness and stability of solutions with  periodic inputs were considered. Finally, we presented several illustrative numerical cases in line with the developed analysis. A less conservative estimate of the error bound for periodic inputs, as well as its dependence on the gain parameter and the Lipschitz constant of the input trajectory, could be the subject of future investigations, especially with regard to control applications. 
 
\color{black}

\bibliographystyle{elsarticle-harv}
\bibliography{references}

\begin{thebibliography}{10}
\expandafter\ifx\csname natexlab\endcsname\relax\def\natexlab#1{#1}\fi
\providecommand{\url}[1]{\texttt{#1}}
\providecommand{\href}[2]{#2}
\providecommand{\path}[1]{#1}
\providecommand{\DOIprefix}{doi:}
\providecommand{\ArXivprefix}{arXiv:}
\providecommand{\URLprefix}{URL: }
\providecommand{\Pubmedprefix}{pmid:}
\providecommand{\doi}[1]{\href{http://dx.doi.org/#1}{\path{#1}}}
\providecommand{\Pubmed}[1]{\href{pmid:#1}{\path{#1}}}
\providecommand{\bibinfo}[2]{#2}
\ifx\xfnm\relax \def\xfnm[#1]{\unskip,\space#1}\fi
\bibitem[{Brokate and Sprekels(1996)}]{BrokateSprekels}
\bibinfo{author}{Brokate, M.}, \bibinfo{author}{Sprekels, J.},
  \bibinfo{year}{1996}.
\newblock \bibinfo{title}{Hysteresis and Phase Transitions}.
\newblock \bibinfo{publisher}{Springer}.
\bibitem[{Kopfová(1999)}]{3}
\bibinfo{author}{Kopfová, J.}, \bibinfo{year}{1999}.
\newblock \bibinfo{title}{Uniqueness theorem for a {Cauchy} problem with
  hysteresis}.
\newblock \bibinfo{journal}{Proc. of American Mathematical Society}
  \bibinfo{volume}{127}, \bibinfo{pages}{3527--3532}.
\bibitem[{Krasnosel’skii and Pokrovskii(1989)}]{1}
\bibinfo{author}{Krasnosel’skii, M.}, \bibinfo{author}{Pokrovskii, A.},
  \bibinfo{year}{1989}.
\newblock \bibinfo{title}{Systems with hysteresis}.
\newblock \bibinfo{publisher}{Springer}.
\bibitem[{Krejci(1996)}]{9}
\bibinfo{author}{Krejci, P.}, \bibinfo{year}{1996}.
\newblock \bibinfo{title}{Hysteresis, convexity and dissipation in hyperbolic
  equations}.
\newblock \bibinfo{publisher}{Tokyo: Gakkotosho}.
\bibitem[{Mayergoyz(1991)}]{Mayergoyz}
\bibinfo{author}{Mayergoyz, I.}, \bibinfo{year}{1991}.
\newblock \bibinfo{title}{Mathematical Models of Hysteresis}.
\newblock \bibinfo{publisher}{Springer}.
\bibitem[{Minorowicz et~al.(2016)Minorowicz, Leonetti, Stefanski, Binetti and
  Naso}]{7}
\bibinfo{author}{Minorowicz, B.}, \bibinfo{author}{Leonetti, G.},
  \bibinfo{author}{Stefanski, F.}, \bibinfo{author}{Binetti, G.},
  \bibinfo{author}{Naso, D.}, \bibinfo{year}{2016}.
\newblock \bibinfo{title}{Design, modelling and control of a micro-positioning
  actuator based on magnetic shape memory alloys}.
\newblock \bibinfo{journal}{Smart Materials and Structures}
  \bibinfo{volume}{25}, \bibinfo{pages}{075005}.
\bibitem[{Ruderman(2023)}]{2}
\bibinfo{author}{Ruderman, M.}, \bibinfo{year}{2023}.
\newblock \bibinfo{title}{Inversion-free feedforward hysteresis control using
  {Preisach} model}, in: \bibinfo{booktitle}{IEEE European Control Conference},
  pp. \bibinfo{pages}{1--6}.
\bibitem[{Ruderman and Bertram(2014)}]{5}
\bibinfo{author}{Ruderman, M.}, \bibinfo{author}{Bertram, T.},
  \bibinfo{year}{2014}.
\newblock \bibinfo{title}{Control of magnetic shape memory actuators using
  observer-based inverse hysteresis approach}.
\newblock \bibinfo{journal}{IEEE Transactions on Control Systems Technology}
  \bibinfo{volume}{22}, \bibinfo{pages}{1181--1189}.
\bibitem[{Ruderman et~al.(2025)Ruderman, Giostra and Sette}]{6}
\bibinfo{author}{Ruderman, M.}, \bibinfo{author}{Giostra, G.},
  \bibinfo{author}{Sette, M.}, \bibinfo{year}{2025}.
\newblock \bibinfo{title}{Inversion-free feed-forward and feedback control of
  {MSM} based actuator with large non-smooth input hysteresis}.
\newblock \bibinfo{journal}{IFAC-PapersOnLine: ALCOS2025} .
\bibitem[{Visintin(1994)}]{8}
\bibinfo{author}{Visintin, A.}, \bibinfo{year}{1994}.
\newblock \bibinfo{title}{Differential models of hysteresis} .

\end{thebibliography}

\end{document}